\documentclass[a4paper]{amsart}

\usepackage[utf8]{inputenc}

\usepackage{amsmath}
\usepackage{amssymb}
\usepackage{amsthm}
\usepackage{latexsym}
\usepackage{graphicx}

\theoremstyle{plain}
\newtheorem{thm}{Theorem}[section] 
\newtheorem{lem}[thm]{Lemma} 
\newtheorem{cor}[thm]{Corollary}
\newtheorem{prop}[thm]{Proposition}

\theoremstyle{definition}

\newtheorem{exmp}{Example}[section] 
\newtheorem{rem}{Remark} 

\title{Singular factors of rational plane curves}
\author[Laurent Bus\'e]{Laurent Bus\'e}
\address{INRIA Sophia Antipolis - M\'editerran\'ee.
           2004 route des Lucioles, B.P. 93
           06902 Sophia Antipolis, France}
\email{Laurent.Buse@inria.fr}
\urladdr{\tt http://www-sop.inria.fr/members/Laurent.Buse/}

\author[Carlos D'Andrea]{Carlos D'Andrea}
\address{Departament d'{\`A}lgebra i Geometria, Universitat de Barcelona.
Gran Via 585, 08007 Barcelona, Spain}
\email{cdandrea@ub.edu}
\urladdr{\tt http://atlas.mat.ub.es/personals/dandrea}

\date{\today}
\thanks{Both authors were partially supported by the bilateral (French-Spanish) collaboration PAI Picasso HF 2006--0220 . The second author was also partially supported
by the research project MTM2007--67493 (Spain).}

\def\PP{{\mathbb{P}}}
\def\AA{{\mathbb{A}}}
\def\CC{{\mathbb{C}}}
\def\N{{\mathbb{N}}}

\def\Cc{{\mathcal{C}}}
\def\Res{{\mathrm{Res}}}
\def\SRes{{\mathrm{SRes}}}
\def\zf{{\mathfrak{z}}}
\def\F{{\mathfrak{F}}}
\def\pp{{\mathfrak{p}}}
\def\coker{{\mathrm{coker}}}
\def\Sylv{{\mathrm{Sylv}}}
\def\val{{\mathrm{val}}}

\begin{document}
\begin{abstract}
We give a complete factorization of the invariant factors of resultant matrices built from birational parameterizations of rational plane curves in terms of the 
singular points of the curve and their multiplicity graph. This allows us to prove the validity of some conjectures about these invariants stated by Chen, 
Wang and Liu. As a byproduct, we also give a complete factorization of the $D$-resultant for rational functions in terms of the similar data extracted from the multiplicities.
\end{abstract}

\maketitle
\section{Introduction}\label{1}
Curves in Computer Aided Geometric Design and in Visualization are often given in  parametric form. Their singularities are usually points where the shape of the graphic gets more complicated. Hence, 
understanding the nature and character of these singular points has been an active area of
research in the last years, see for instance \cite{CS01,Par02,SG07,PD07,Chen08,Bus09,JG09} and the references therein.

In this article, we will focus on parametric plane curves defined over the complex numbers, although our results are valid for any field of characteristic zero, 
and the computational aspects can be performed also on any  field containing the coefficients of the input polynomials. 
Let $a,b,c\in\CC[s,v]$ be homogeneous polynomials of the same degree 
$n\geq 3$  with $\gcd(a,b,c)=1$, such that the map
\begin{equation}\label{eq:phi}
\begin{array}{cccc}
\phi:&\PP_\CC^1 &\rightarrow &\PP^2_\CC \\
&(s_0:v_0)& \mapsto &(a(s_0,v_0):b(s_0,v_0):c(s_0,v_0))	
\end{array}
\end{equation}
parameterizes a plane rational algebraic curve $\Cc$ birationally onto its image (which is equivalent to say that the degree of $\Cc$ is $n$).  As it was shown by Abhyankar in \cite{Abh90} for $c=v^n$, and  later in general by Sendra and Winkler in \cite{SW01}, Gutierrez, Rubio and Yie  in \cite{GRY02}, and P\'erez-Diaz in \cite{PD07} among others, the parameterization $\phi$ can be used to detect the singular points of $\Cc$, which are those  $P\in\Cc$ such that their multiplicity $m_P(\Cc)$ is strictly greater than $1$. As $\phi$ is generically one-to-one, $m_P(\Cc)$ is actually the number of points in the preimage of $\phi^{-1}(P)$ counted with multiplicities (for a proper ``parameterization-free'' definition of $m_P(\Cc)$ as well as its properties, see \cite{Abh90,Wal50}). This explains why from a computational point of view,  the
parameterization $\phi$ provides a lot of information about the singularities of $\Cc$. The purpose
of this paper is to shed some light in this area. 
\par\smallskip
We will use the notation and definitions given in \cite{BK86} (see also \cite{Sta00}).
Let $\{P_1,\ldots,P_r\}$ be the proper singular points of $\Cc$, and for all $i=1,\ldots,r$ denote by 
\begin{itemize}
	\item $\zf^i_j$, $j\in I_i$, the irreducible branch-curves of $\Cc$ at $P_i$,
	\item $(t_{i,j}:u_{i,j})$, $j\in I_i$, the point of $\PP^1_\CC$ such that $\zf^i_j(t_{i,j}:u_{i,j})=P_i$,
	\item $(P^i_{j,h})_{0\leq h}$ the neighboring point sequence of $\zf^i_j$ at $P_i$,
	\item $(m^i_{j,h})_{0\leq h}$ the multiplicity sequence of $\zf^i_j$ at $P_i$,
	\item $(\sim_h)_{o\leq h}$ the equivalence relations of the multiplicity graph of $\Cc$.
\end{itemize} 
For a virtual point $P^i_{j,h}$ of $\Cc$, we define its multiplicity as
$m_{P^i_{j,h}}(\Cc):=\sum_{j'\sim_h j} m^i_{j',h}$.
Set
\begin{equation}\label{FG}
\begin{array}{ccc}
F(s,v;t,u)&:=&a(s,v)c(t,u)-a(t,u)c(s,v)\\
G(s,v;t,u)&:=&b(s,v)c(t,u)-b(t,u)c(s,v),
\end{array}
\end{equation}
and let
$\Res_{(s,v)}(-,-)$ be the Sylvester resultant operator which eliminates the homogeneous variables $s$
and $v$ . If $c=v^n$, then it is shown in \cite{Abh90} that there exists $0\neq \gamma \in \CC$  such that
\begin{equation}\label{tayloresultant}
\Res_{(s,v)}\left(\frac{F(s,v;t,u)}{su-tv},\frac{G(s,v;t,u)}{su-tv}\right) = \gamma \prod_{\substack{i=1,\ldots,r \\ j \in I_i}} (u_{i,j}t-t_{i,j}u)^{\epsilon_{i,j}}
\end{equation}
where for all $i=1,\ldots,r$ and $j\in I_i$
\begin{equation*}\epsilon_{i,j} = \sum_{h\geq 0} m^i_{j,h}(m_{P_{j,h}^h}(\Cc)-1).
\end{equation*}
From now on we will most of the time omit the nonzero constants. Hence, all identities involving polynomials should be understood up to a nonzero $\gamma\in\CC.$
\par
Let $B_{F,G}(t,u)\in\CC[t,u]^{n\times n}$ be the square B\'ezout matrix built from
$F(s,v;t,u),$ $\,G(s,v,t;u)$ regarded as polynomials in the variables $(s,v)$  (see Section \ref{5} for its precise definition and construction). Clearly, $B_{F,G}(t,u)$ does not have
maximal rank as $su-tv$ is a common factor of both $F$ and $G$. In \cite{CS01}, Chiohn and Sederberg showed that by analyzing the maximal minors of this matrix,  one can obtain all the singular points of $\Cc$ in a very direct way.
This approach was improved and refined by Chen, Wang and Liu in \cite{Chen08}, where 
(\ref{FG}) is replaced with
 \begin{equation}\label{pq}
\begin{array}{ccc}
	p_\phi(s,v;t,u) &= &p_1(s,v)a(t,u)+p_2(s,v)b(t,u)+p_3(s,v)c(t,u) \\ 
	q_\phi(s,v;t,u) &= &q_1(s,v)a(t,u)+q_2(s,v)b(t,u)+q_3(s,v)c(t,u), 
\end{array}
\end{equation}
with $\{p,q\}:=\{(p_1,p_2,p_3),\,(q_1,q_2,q_3)\}$ being a basis of the free $\CC[s,v]$-module of syzygies of $(a,b,c)$. Suppose w.l.o.g. that $\deg(p)\leq\deg(q)$ and set
$\mu:=\deg(p)$. We then easily have $\mu\leq n-\mu=\deg(q)$. In the Computer Aided Geometric Design community, the set $\{p,q\}$ is called a \text{$\mu$-basis} of the parameterization $\phi$.
\par\smallskip
Let now $B_{p_\phi,q_\phi}(t,u)\in\CC[t,u]^{(n-\mu)\times(n-\mu)}$ be the  hybrid B\'ezout matrix associated to $p_\phi(s,v;t,u),\,q_\phi(s,v;t,u)$ (for a definition of hybrid B\'ezout matrices, se also Section \ref{5}).
It is shown in \cite{Chen08} that by computing the invariant factors of this matrix, one gets some kind of stratification of the singularities of $\Cc$ 
with respect to their multiplicities (Theorem
$4$ in \cite{Chen08}). This stratification is well understood when all the singularities of $\Cc$ are
\textit{ordinary} (i.e. when there are no virtual points $P^i_{j,h}$  with $h>0$) and one can
get an explicit description of the invariant factors of this matrix in terms of the singular points of $\Cc$ and their multiplicities (Theorems $5$ and $6$ in \cite{Chen08}). 
\par
In the case where $\Cc$ has  singularities that are not ordinary, Chen, Wang and Liu stated a couple of conjectures
(Conjectures $1$ and $2$ in \cite{Chen08}) relating the invariant factors with the multiplicity of the virtual
points appearing in the process of desingularization of the curve. The main result of this paper is
a complete factorization of the singular factors of $B_{p_\phi,q_\phi}(t,u)$ and as a consequence a complete proof and clarification of both conjectures. 
\par
To be more precise, let  $S_{p_\phi,q_\phi}(t,u)$ be the Sylvester matrix of (\ref{pq}). It is simply the square $(n\times n)$-matrix of the $\CC[t,u]$-linear map 
\begin{eqnarray}\label{eq:sylvestermatrix}
	\CC[t,u]\otimes_\CC \CC[s,v]_{n-\mu-1} \oplus \CC[t,u]\otimes_\CC \CC[s,v]_{\mu-1} & \rightarrow & \CC[t,u]\otimes_\CC \CC[s,v]_{n-1} \\ \nonumber
	 \alpha \oplus \beta  & \mapsto & \alpha p_\phi+ \beta q_\phi
\end{eqnarray} 
in the canonical monomial bases (the notation $\CC[s,v]_{d}$, $d\in \mathbb{N}$, stands for the $\CC$-vector space of homogeneous polynomials of degree $d$ in $\CC[s,v]$).
 A collection of homogeneous polynomials $d_1(t,u),d_2(t,u),\ldots,$ $d_n(t,u)$ in $\CC[t,u]$ such that, for $i=1,\ldots,n$ the product 
$$d_n(t,u)^{n-i+1}d_{n-1}(t,u)^{n-i}\cdots d_{i+1}(t,u)^2d_{i}(t,u) \in \CC[t,u]$$ 
is equal to the greatest common divisor of the $(n+1-i)$-minors of $S_{p_\phi,q_\phi}(t,u)$, is called a collection of \emph{singular factors} of the parameterization $\phi$. The existence of these polynomials is guaranteed by  homogenizing with some care the invariant factors of $S_{p_\phi,q_\phi}(t,1)$.

The terminology of \emph{singular factors} is taken from \cite{Chen08}. Note also that if $a,b,c\in k[s,v]$ with $k$ a subfield of $\CC$, then the singular factors will have their coefficients in $k$. This observation is of importance for computational purposes.

\medskip

Now we are ready to present the main result of this paper. 

\begin{thm}\label{thm:main} With the notation established above, we have 
	$$d_{n-\mu+1}(t,u)=\cdots=d_n(t,u)=1,$$ 
	and for  $k=2,\ldots,n-\mu$  
	$$  d_k(t,u)= \prod_{i=1,\ldots,r, \ j \in I_i} (u_{i,j}t-t_{i,j}u)^{\epsilon^k_{i,j}}$$
	where
	$$\epsilon^k_{i,j}=\sum_{h \text{ \rm such that } m_{P^i_{j,h}}(\Cc)=k} m^i_{j,h}$$
\end{thm}

We will see how this theorem implies Conjectures $1$ and $2$ in \cite{Chen08} in Section \ref{2} and prove it in Section \ref{4}.
We also point out that the factorization of invariant factors of matrices related to this problem is also considered in \cite{JG09}.

\par
The reader may have already noticed that we just claimed above that the conjectures posted in \cite{Chen08}  where made over the matrix 
 $B_{p_\phi,q_\phi}(t,u)$ instead of $S_{p_\phi,q_\phi}(t,u)$. We will show in Section \ref{5} that the cokernels of these two matrices plus a whole family of hybrid resultant matrices are isomorphic, thus Theorem \ref{thm:main} can also be formulated over the invariant factors of any of them. Our results in Section \ref{5} can be regarded as an extension of those shown already by Ap\'ery and Jouanolou in \cite[Proposition 18]{AJ06}.

\par
We will also see in Section \ref{6} that there is an explicit connection between $B_{p_\phi,q_\phi}(t,u)$ and $B_{F,G}(t,u)$, which will allow us to get a complete description of the invariant factors of the latter.
As a direct consequence of this, we get a complete factorization into irreducible factors of 
$D$-resultants. These are a natural generalization of Abhyankar's formula (\ref{tayloresultant}) for $c=v^n$. Indeed, in \cite{GRY02}, it is shown that  if we take
\begin{equation}\label{eqq}
\tilde{\Delta}(t,u):=\Res_{(s,v)}\left(\frac{F(s,v;t,u)}{su-tv},\frac{G(s,v;t,u)}{su-tv}\right),
\end{equation}
for a general rational parameterization, it turns out that if $\phi(t_0:u_0)$ is a singularity of $\Cc$, then $\tilde{\Delta}(t_0,u_0)=0,$
but there may be other roots coming from curves being parameterized by  permutations of $(a,b,c)$ see \cite[Theorem $3.1$]{GRY02}, and there was no known analogue of a factorization like (\ref{tayloresultant}) for $\tilde{\Delta}(t,u)$.
\par In \cite{Bus09}, the first author shows that by replacing (\ref{FG}) with (\ref{pq}) in the definition of $\tilde{\Delta}(t,u)$, one gets a polynomial $\Delta(t,u)$ which factorizes like (\ref{tayloresultant}). We will review the properties of $\Delta(t,u)$ in Section \ref{2}.
\par However, there was still missing a complete factorization of $\tilde{\Delta}(t,u)$. In Theorems \ref{thm:dres}  and \ref{thm:diffd}  we  give   a precise description of the factors of the $D$-resultant, completing the information given in
\cite[Theorem $3.1$]{GRY02}.

\smallskip
Understanding the algebraic structure of these matrices may lead to new algorithms for studying the geometry of singular points of rational curves. We will see for instance in Example \ref{example} that in some non trivial cases one can reconstruct the whole multiplicity graph of $\Cc$ from the invariant factors of these matrices, with operations only over the ground field of the parameterization. From a symbolic point of view, this problem has already been studied in \cite{SW91, Sta00,PD07}.

\subsection*{Organization of the paper.}
In Section \ref{2} we introduce some basic definitions and results, and also show how  Theorem \ref{thm:main} implies Conjectures $1$ and $2$ in \cite{Chen08}.
In Section \ref{3} we prove the main theorem for  curves having only ordinary singularities. The proof of the general case is given in Section
\ref{4}. 
Section \ref{5} is devoted to show that any resultant matrix can be used in Theorem \ref{5}. In Section \ref{6}, we describe all the invariant factors of $B_{F,G}(t,u)$ and show the complete factorization of $D$-resultants.

\subsection*{Acknowledgements.}
We would like to thank Jos\'e Ignacio Burgos, Eduardo Casas-Alvero and Teresa Cortadellas for very interesting comments, suggestions and clarifications on topics about singularities of curves and commutative algebra.  
The second author would also like to thank the Fields Institute in Toronto, where part of this work was done.

\section{Preliminary results and the singular factors conjectures}\label{2}
Throughout this paper, we will work over the field of complex numbers $\CC$. However, it should be noted that all the statements and proofs work over any algebraically closed field of characteristic zero.
We recall here again that  every identity involving polynomials should be understood up to a nonzero constant.
\par
Let $(x_1:x_2:x_3)$ be the homogeneous coordinates of $\PP^2_\CC$. By Hilbert-Burch's theorem, the first syzygy module of the sequence $a(s,v),b(s,v),c(s,v)$ is a free $\CC[s,v]$-module of rank 2. Moreover, if $\mu$ denotes the smallest degree of a nonzero syzygy, then this syzygy module is generated in degrees $\mu$ and $n-\mu$. A $\mu$-basis of the parameterization $\phi$ is then a choice of a basis of this syzygy module. Identifying any syzygy $(g_1,g_2,g_3)$ with the bi-homogeneous form $g_1x_1+g_2x_2+g_3x_3$, a $\mu$-basis  corresponds to a couple of bi-homogeneous forms  
$p,q \in \CC[s,v]\otimes_\CC \CC[x_1,x_2,x_3]$ of bi-degree $(\mu,1)$ and $(n-\mu,1)$ respectively, such that $1\leq \mu\leq n-\mu$. 

\medskip

As $\Cc$ is a rational projective curve, we have that  its number of singular points, counted properly, is given by the well known \emph{genus formula}:
$$ (n-1)(n-2)=\sum_{P\in \mathrm{Sing}(\Cc)} m_P(\Cc)(m_P(\Cc)-1)$$
where $\mathrm{Sing}(\Cc)$ stands for the singular points, proper as well as infinitely near, of the curve $\Cc$ and $m_P(\Cc)$ stands for the multiplicity of the singular point $P$ on $\Cc$. Notice that 
in our case we know that there exists at least one (proper) singular point on $\Cc$, since $n\geq 3$.

\medskip
It is a well  known fact that $\Res_{(s,v)}(p,q) \in \CC[x_1,x_2,x_3]$  is an implicit equation of the curve $\Cc$, meaning that it is an irreducible and homogeneous degree $n$ polynomial whose zero locus is exactly the curve $\Cc$ (recall that the parametrization $\phi$ is assumed to be birational onto $\Cc$). Another interesting property is the following (see also \cite[Lemma 2]{Chen08}):

\begin{prop}\label{prop:inversionformula} Let $Q=(\alpha_1:\alpha_2:\alpha_3)$ be a point in $\PP_\CC^2$ and denote by $H_{Q}(s,v)$ a greatest common divisor of the two forms $\sum_{i=1}^3\alpha_ip_i(s,v)$ and $\sum_{i=1}^3\alpha_i q_i(s,v)$ in $\CC[s,v]$. Then $H_{Q}(s,v)$ is a homogeneous polynomial of degree $m_Q(\Cc)$ and if $m_Q(\Cc) \geq 1$ we have
	$$ H_Q(s,v)=\prod_{i=1}^{N}(v_is-s_iv)^{m_i}$$
where $N$ is the number of irreducible branch-curves of $\Cc$ centered at $Q$ and $m_i$ denotes the multiplicity of $Q$ with respect to the irreducible branch-curve $\zf$ such that $\zf(s_i:v_i)=Q$.
\end{prop}
\begin{proof} By a linear change of coordinates in $\PP^2_\CC$, one can assume that $Q=(0:0:1)$, because $\mu$-bases have the expected property under linear changes of coordinates. By Hilbert Burch's theorem we have that, up to a constant,
	$$a(s,v)=\left| 
	\begin{array}{cc}
		p_2(s,v) & p_3(s,v) \\
		q_2(s,v) & q_3(s,v)
	\end{array}
	\right|
	\textrm{ and }
	b(s,v)=\left| 
	\begin{array}{cc}
		p_3(s,v) & p_1(s,v) \\
		q_3(s,v) & q_1(s,v)
	\end{array}
	\right|.
$$	
If $h(s,v)$ is a divisor of $p_3(s,v)$ and $q_3(s,v)$, then it follows that $h(s,v)$ divides both $a(s,v)$ and $b(s,v)$. Reciprocally, suppose that $h(s,v)$ is a nontrivial divisor of $a(s,v)$ and $b(s,v)$. Then, as
$$p_1(s,v)a(s,v)+p_2(s,v)b(s,v)+p_3(s,v)c(s,v)=0$$
and $c(s,v)$ does not share any nontrivial common factor with neither $a(s,v)$ nor $b(s,v)$, then $h(s,v)$ must divide $p_3(s,v)$. The same argument works for $q_3(s,v)$ and we deduce that $\gcd(p_3(s,v),q_3(s,v))=\gcd(a(s,v),b(s,v))$. From here, the claimed equality follows from the definition of the multiplicity of a point on an irreducible branch-curve. 
\end{proof}
\begin{rem}\label{rmm}
As an easy consequence of Proposition \ref{prop:inversionformula}, we have that if $Q$ is a proper singular point on $\Cc$ then either $2\leq m_Q(\Cc)\leq \mu$ or $m_Q(\Cc)= n-\mu$, a fact that has already been noticed in \cite{SG07}.
\end{rem}

Proposition \ref{prop:inversionformula} shows that a $\mu$-basis of $\phi$ provides nontrivial information on the proper singularities of $\Cc$. It turns out that it also carries informations on the infinitely near singularities of $\Cc$. 
Recall that we have the following property:
$$m_{P^i_{j,h}}(\Cc)=\sum_{j'\sim_h j} m^i_{j',h} \geq m^i_{j,h}.$$
Also, denote with $\SRes(p,q) \in \CC[x_1,x_2,x_3]$ the first principal subresultant of $p$ and $q$ with respect to the couple of homogeneous variables $(s,v)$. It is simply a certain minor of the Sylvester matrix of $p$ and $q$ with respect to $(s,v)$ (see e.g.~\cite{AJ06,Elk05,Bus09} for a precise definition).
The following result is a slight refinement of  \cite[Section 4]{Bus09}.

\begin{thm}\label{thm:conductor} We have 
	\begin{equation}\label{eq:delta}
		\Delta(t,u)=\SRes(p,q)(a(t,u),b(t,u),c(t,u)) = \gamma \prod_{\substack{i=1,\ldots,r \\ j \in I_i}} (u_{i,j}t-t_{i,j}u)^{\epsilon_{i,j}}
	\end{equation}
where $0\neq \gamma \in \CC$ and for all $i=1,\ldots,r$ and $j\in I_i$
\begin{equation*}\epsilon_{i,j} = \sum_{h\geq 0} m^i_{j,h}(m_{P_{j,h}^h}(\Cc)-1).
\end{equation*}
In particular, 
$$\deg(\Delta(t,u))= (\deg(\Cc)-1)(\deg(\Cc)-2)=\sum_{P\in \mathrm{Sing}(\Cc)} m_P(\Cc)(m_P(\Cc)-1).$$
\end{thm}
\begin{proof} Although not stated explicitly under this form, this theorem follows from the results contained in \cite[Section 4]{Bus09}. Indeed, the proof of Theorem 4.8 in loc.~cit.~shows that all the inequalities given in Proposition 4.6, always in loc.~cit., are actually equalities for the first principal subresultant of $p$ and $q$. Since these equalities occur at the level of irreducible branch-curves, they imply the above theorem which requires the use of the multiplicity graph of $\Cc$. In fact, the properties (P1), (P2) and (P3) in loc.~cit.~are consequences of this theorem but they constitute the finer result one can state without introducing the multiplicity graph of $\Cc$. 
\end{proof}

Notice that Theorem \ref{thm:main} can be regarded as a refinement of the above theorem, since it provides non trivial factors  of $\eqref{eq:delta}$  that are in relation with the multiplicities of the singular points of the curve $\Cc$.

\medskip
Let us now show  how Theorem \ref{thm:main} implies the two conjectures stated in \cite{Chen08}.
For each proper singular point $P_i \in \Cc$, $i=1,\ldots,r$, we have
$$H_{P_i}(t,u)=\prod_{j\in I_i} (u_{i,j}t-t_{i,j}u)^{m_{P_i}(\zf_j^i)}=\prod_{j\in I_i} (u_{i,j}t-t_{i,j}u)^{m^i_{j,0}}.$$
For all $2\leq k \leq n-\mu$, set
$$h_k(t,u)=\prod_{P_i \text{ such that } m_{P_i}(\Cc)=k} H_{P_i}(t,u).$$
From Theorem \ref{thm:main}, it is clear that $h_k(t,u)$ divides $d_k(t,u)$. Actually, we have:
$$  d_k(t,u)=  h_k(t,u) \prod_{\substack{i=1,\ldots,r \\ j \in I_i}} (u_{i,j}t-t_{i,j}u)^{\overline{\epsilon}^k_{i,j}},$$
where
$$\overline{\epsilon}^k_{i,j}=\sum_{h>0 \text{ \rm such that } m_{P^i_{j,h}}(\Cc)=k} m^i_{j,h}=\epsilon^k_{i,j} - 
\begin{cases}
	m^i_{j,0} & \text{ if } m_{P_i}(\Cc)=k\\
	0 & \text{ otherwise. }
\end{cases}
$$
Therefore, following the notation in \cite[Conjecture 1]{Chen08}, we can establish the validity of the
first conjecture: 
$$d_k(t,u)=h_k(t,u)\prod_{s=k}^{n-\mu} \Psi_k^s(t,u),$$
where for all pair $2 \leq k,s \leq n-\mu$,
$$\Psi_k^s(t,u)= \prod_{\substack{i=1,\ldots,r \\ j \in I_i \\ m_{P_i}(\Cc)=s}} (u_{i,j}t-t_{i,j}u)^{\overline{\epsilon}^k_{i,j}}.$$
Obviously $\Psi_k^s(t,u)=0$ if $s<k$ (for multiplicities cannot increase through blowing up). Moreover, it is not hard to check that $\deg(\Psi_k^s(t,u))$ is $k$ times the number of infinitely near and non-proper singularities of multiplicity $r$
above a proper singular point of multiplicity $s$, which proves and makes more precise \cite[Conjecture 1]{Chen08}.

\medskip

Now, define for all $k=2,\ldots,n-\mu$  the \emph{reduced singular factor} $\tilde{d_k}(t,u)$ by the following procedure: 
\begin{itemize}
	\item Set $\tilde{d_k}(t,u):=d_k(t,u).$
	\item Then, for all $l=n-\mu$ down to $k+1$ do $\tilde{d_k}(t,u):=\frac{\tilde{d_k}(t,u)}{\gcd(\tilde{d_k}(t,u),d_l(t,u))}$.
\end{itemize}
Theorem \ref{thm:main} then implies that
$$  \tilde{d_k}(t,u)= \prod_{i=1,\ldots,r, \ j \in I_i} (u_{i,j} t-t_{i,j} u)^{\tilde{\epsilon}^k_{i,j}},$$
where
$$\tilde{\epsilon}^k_{i,j}=\max\{ \epsilon^k_{i,j} - \sum_{s=k+1}^{n-\mu}\epsilon^s_{i,j}  , 0\}.$$
Therefore $\tilde{\epsilon}^k_{i,j} \neq 0$ if and only if $m_{P_i}(\Cc)=k$ and in this case it is equal to $\epsilon^k_{i,j}$. It follows that 
$$  \tilde{d_k}(t,u)= \prod_{\substack{i=1,\ldots,r \\  j \in I_i \\ m_{P_i}(\Cc)=k}} (u_{i,j} t-t_{i,j} u)^{\epsilon^k_{i,j}} =\prod_{P_i \text{ such that } m_{P_i}=k} H_{P_i}(t,u)^{l_i},$$
where $l_i$ is the number of infinitely near points of multiplicity $k$ above $P_i$, including $P_i$. This proves \cite[Conjecture 2]{Chen08}.

\medskip

Before moving on to the next section, from a computational  as well as theoretical point of view, it is interesting  to point out  that $\CC[t]$ is a principal ideal domain, and hence one can use the theory of invariant factors over principal domains in order to get that the matrix $S_{p_\phi,q_\phi}(t,1)$ is equivalent to a diagonal 
matrix whose nonzero elements are 
$$d_{n}(t,1),d_n(t,1)d_{n-1}(t,1),\cdots,d_n(t,1)d_{n-1}(t,1)\ldots d_3(t,1)d_2(t,1),0.$$
We will recall and use this property for proving Theorem \ref{thm:main}. Notice also that a \emph{single} Smith normal form computation of $S_{p_\phi,q_\phi}(t,1)$ yields {\em all} the singular factors of $\phi$, after a linear change of coordinates of $\PP^1$ if necessary -- see Lemma \ref{lem:chgtcoord2}. Let us conclude this section with an illustrative example.

\begin{exmp}\label{example}
	Take the following parameterization of a rational algebraic plane curve of degree $n=10$: 
	$$\left\{
	\begin{array}{lcl}
		a & = & {s}^{2} \left( 2\,s+v \right) ^{2} \left( s + v \right) ^{6} \\
		b & = & {s}^{3} \left( 2\,s+v
		 \right) ^{5} \left(3\,{s}^{2}+ 2\,sv +{v}^{2} \right)   \\
		c & = & - \left( s+v \right) ^{10}
	 \end{array}\right.$$
The computation of the $\mu$-basis gives $\mu=4$ and
	$$p = (s+v)^4x_1+{s}^{2} \left( 2\,s+v \right) ^{2}x_3$$
	$$q = s \left( 3\,{s}^{2}+2\,sv+{v}^{2}\right) \left( 2\,s+v \right) ^{3} x_1 - \left( s+v \right) ^{6}x_2 $$
The associated  B\'ezout matrix  is then a $6\times 6$-matrix from which we get, after dehomogenization $u=1$ and a single Smith form computation, the following singular factors
	$$d_6(t)=(t+1)^6, \ d_5(t)=1, \ d_4(t)=\frac{1}{4} \left( 2\,t+1 \right) ^{2} \left( t+1 \right) ^{4}{t}^{2}, \ d_3(t)=\frac{1}{4} \left( 2\,t+1 \right) ^{2}t,$$
	$$d_2(t)=\frac{1}{43} \left( 43\,{t}^{6}+74\,{t}^{5}+71\,{t}^{4}+48\,{t}^{3}+21\,{t}^
	{2}+6\,t+1 \right)  \left( t+1 \right) ^{6} $$ 
and reduced singular factors
		$$\tilde{d}_6(t)=d_6(t), \ \tilde{d}_5(t)=d_5(t)=1, \tilde{d}_4(t)= \frac{1}{4} \left( 2\,t+1 \right) ^{2} {t}^{2}, \tilde{d}_3(t)=1,$$
		$$\tilde{d}_2(t)=\frac{1}{43} \left( 43\,{t}^{6}+74\,{t}^{5}+71\,{t}^{4}+48\,{t}^{3}+21\,{t}^
		{2}+6\,t+1 \right)$$
Although it is not always possible in general, we can recover here the multiplicity graph of the curve by using Theorem \ref{thm:main}. For that purpose, we start by inspecting $d_6$, the non-trivial singular factor with the highest index. We deduce that there is an irreducible singularity of multiplicity 6 corresponding to the parameter value $t=-1$. Looking at the other singular factors, we obtain that this singular point has a multiplicity 4 singular point in its first neighborhood and then has singular points of multiplicity 2 in  its third, fourth and fifth neighborhood. So we obtain the first branch of the multiplicity graph, see Fig.~\ref{multgraph}.

Now, by inspecting $d_4$, we deduce that there is a singular point of multiplicity 4 which is formed by two irreducible branch-curves, one, say $\zf_1$ centered at $t=-1/2$ and another one, say $\zf_2$ centered at $t=0$. The singular factor $d_3$ then shows that these two irreducible branches split up at the third neighborhood and have the multiplicities given in Fig.~\ref{multgraph} in the second neighborhood (the horizontal bar stands for the equivalence relation of the multiplicity graph). Then, since $d_2$ does not vanish at $t=-1/2$ or $t=0$ the mutliplicity graph at this multiplicity 4 point is known; see Fig.~\ref{multgraph}. 

Finally, a simple additional computation shows that the discriminant of $d_2/(t+1)^6$ is nonzero. Therefore, it only remains to add 3 ordinary double points to the multiplicity graph to complete it.

\begin{figure}[!h]
	\begin{center}
	\includegraphics{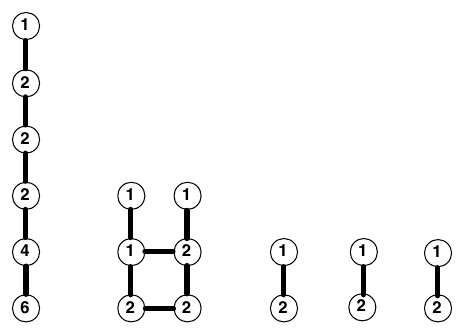}	
	\end{center}
\caption{Multiplicity graph of the degree 10 rational plane curve given in Example \ref{example}}
\label{multgraph}	
\end{figure}
\end{exmp}

\section{The case of ordinary curves}\label{3}

To prove Theorem \ref{thm:main}, we will proceed by induction on the minimal length of a
resolution of singularities of $\Cc$. The initial step would then correspond to the case where $\Cc$ is an ordinary curve, i.e. $\Cc$ has only ordinary singularities.  Although a proof of this result already appeared in \cite[Theorem 5]{Chen08}, we provide in this section an alternative proof for the sake of completeness and also as the preparation for the proof of Theorem \ref{thm:main}. 

\medskip

We start by recalling very briefly some results of invariant factors and Fitting ideals theory we will need in the sequel. The reader may consult any book on Algebra like \cite[Chapter III, \S 7 and Chapter XIX, \S2]{Lang} for proofs of these statements.

Let $R$ be a principal ideal domain and $M$  a finitely generated $R$-module. There exists a sequence of non invertible elements $(\alpha_1,\ldots,\alpha_\ell)$ such that
	\begin{itemize}
		\item[\rm i)] For all $i=1,\ldots, \ell-1$, $\alpha_i$ divides $\alpha_{i+1}$.
		\item[\rm ii)] $M$ is isomorphic to $\frac{R}{(\alpha_1)}\oplus \frac{R}{(\alpha_2)}\oplus\cdots 
\oplus\frac{R}{(\alpha_\ell)}$. 
	\end{itemize}
The elements $\alpha_1,\ldots,\alpha_\ell$ are unique up to multiplication by a unit and are called the \emph{invariant factors} of the $R$-module $M$. They can be recovered from the Fitting ideals of $M$, denoted $\F_i(M)$, since
\begin{multline*}
	\F_0(M)=(\alpha_1\ldots\alpha_\ell)\subset \F_1(M)=(\alpha_1\ldots\alpha_{\ell-1})\subset \ldots\subset \F_{\ell-1}(M)=(\alpha_1) \subset \\ 
	\F_\ell(M)=\F_{\ell+1}(M)=\cdots=R	
\end{multline*}
The smallest integer $r$ such that $\F_r(M)\neq 0$ is called the \emph{rank} of M. 

It is important to notice for further use that the Fitting invariants of $M$ commute with localization: if $S$ is a multiplicatively closed subset of $R$ not containing the zero element, then for every integer $\nu\geq 0$ we have
\begin{equation}\label{eq:localFitting}
	\F_\nu(M)R_S=\F_\nu(M_S)	
\end{equation}
Also, the Fitting invariants of $M$ can be computed from any finite $R$-presentation of $M$. Such a presentation corresponds to a matrix, say $A$, with entries in $R$. The above results mean that this matrix is equivalent to a diagonal matrix, sometimes called the Smith normal form of $A$, whose nonzero elements are the invariant factors of $M=\coker(A)$.

We will also use later in the text the following result which is due to Thompson \cite{Thompson82}.
\begin{thm}\label{thompson}
Let $A,B,C$ be three square matrices with entries in $R$ such that $AB=C$. If 
$\alpha_1 | \alpha_2 | \ldots | \alpha_n$, $\beta_1 | \beta_2 | \ldots |\beta_n$, $\gamma_1 | \gamma_2 | \ldots | \gamma_n$ are the invariant factors of $A,B$, and $C$ respectively, then 
$$ \alpha_{i_1}\alpha_{i_2}\cdots\alpha_{i_m}\beta_{j_1}\beta_{j_2}\cdots\beta_{j_m} | \gamma_{i_1+j_1-1}\gamma_{i_2+j_2-2}\cdots \gamma_{i_m+j_m-m}$$
whenever the integer subscripts satisfy
$$1 \leq i_1 < i_2 < \cdots < i_m, \  1 \leq j_1 < j_2 < \cdots < j_m, \ \ i_m+j_m \leq m+n.$$
\end{thm}

Now, we examine the behavior of the singular factors under linear change of coordinates, it will also be very useful in the sequel.

\begin{lem}\label{lem:chgtcoord1} The singular factors of a proper parameterization do not depend neither on the choices of the $\mu$-basis nor the coordinates of $\PP^2$.
\end{lem}
\begin{proof} Indeed, a change of $\mu$-bases or a change of coordinates of $\PP^2$ correspond to elementary transformation of the Sylvester matrix $S_{p_\phi,q_\phi}(t,u)$. Therefore, its Fitting invariants remains unchanged so that the same holds for their codimension one part. 
\end{proof}

Recall that a linear change of coordinates of $\PP^1$, that is to say an isomorphism
$$\psi : \PP^1 \xrightarrow{\sim} \PP^1 : (t':u') \mapsto (\alpha t'+\beta u':\delta t'+\gamma u')$$ 
where $\alpha,\beta,\delta,\gamma \in \CC$ and $\alpha\gamma-\beta\delta\neq 0$, corresponds to the isomorphism of graded rings (``base change'' map)
\begin{eqnarray*}
	\psi^\sharp : \CC[t,u] & \xrightarrow{\sim} & \CC[t',u'] \\
	 t & \mapsto & (\gamma t' - \beta u' )/(\alpha\gamma-\beta\delta)\\
	u & \mapsto & (\delta t' - \alpha u' )/(\alpha\gamma-\beta\delta).
\end{eqnarray*}
The following lemma shows that computing singular factors ``commutes'' with linear changes of coordinates of $\PP^1$. 
\begin{lem}\label{lem:chgtcoord2} Let $d_1(t,u),d_2(t,u),\ldots,$ $d_n(t,u)$ be the singular factors of the parameterization $\phi$ of $\Cc$ and let $\psi$ be a linear change of coordinates in $\PP^1$. Then, the homogeneous polynomials in $\CC[t',u']$
	$$\psi^\sharp(d_1(t,u)),\psi^\sharp(d_2(t,u)),\ldots, \psi^\sharp(d_n(t,u))$$
	are the singular factors of the parameterization $\phi\circ\psi:\PP^1\rightarrow \PP^2$ of $\Cc$.
\end{lem}
\begin{proof} It is not hard to check that if $\{p,q\}$ is a $\mu$-basis of $\phi$ then $\psi^\sharp(p),\psi^\sharp(q)$ is a $\mu$-basis of $\phi\circ\psi$. Now, the map $\psi^\sharp$ gives a  $\CC[t,u]$-module structure to $\CC[t',u']$ so that the Sylvester matrix obtained by the change of coordinates $\psi$ is nothing but $S_{p_\phi,q_\phi}(t,u)\otimes_{\CC[t,u]} \CC[t',u']$.  Therefore, by the right-exactness of tensor product we deduce that the Fitting ideals of the cokernel of $S_{p_\phi,q_\phi}(t,u)\otimes_{\CC[t,u]} \CC[t',u']$ are equal to the Fitting ideals of $S_{p_\phi,q_\phi}(t,u)$ tensored with $\CC[t',u']$ over $\CC[t,u]$. It follows that the same property holds for the codimension one part of these Fitting ideals. Hence, the lemma is proved.
\end{proof}

Let $M$ stand for the cokernel of the Sylvester matrix $S_{p_\phi,q_\phi}(t,1)$ defined in \eqref{eq:sylvestermatrix}. Note that $M$ is a $\CC[t]$-module.

\begin{prop}\label{prop:F1M} With above notation, we have
	$\F_0(M)=0$ and $\F_1(M)=(\Delta(t,1))$; in particular, $M$ has rank 1. 
	Moreover,  $\F_\ell(M)=R$ for all $\ell> n-\mu$.
\end{prop}
\begin{proof} Since \eqref{eq:sylvestermatrix} provides a finite presentation of $M$, one can compute the Fitting invariants of $M$ as the determinantal ideals of \eqref{eq:sylvestermatrix}, namely the Sylvester matrix of $p_\phi$ and $q_\phi$. Thus,  
$\F_0(M)$ is generated by the determinant of this Sylvester matrix, which is equal to zero.
\par 
Recall the following classic property of the Sylvester matrix: the corank of the Sylvester matrix of two given polynomials is equal to the degree of the gcd of these two polynomials. 
So, by using Proposition \ref{prop:inversionformula} and Remark \ref{rmm}, we have that the Fitting invariants of $M$ are supported on the singular locus of  $\Cc$ and 
$\F_\ell(M)=A$ for all $\ell> n-\mu$ since there are no singular points on $\Cc$ of multiplicity $>n-\mu$.
	
	To prove that $\F_1(M)=(\Delta(t,1))$,  we proceed as follows: let
	$$\begin{array}{ccc}
	P(s,v;t)&=&\frac{p_\phi(s,v;t,1)}{s-t\,v}\\
	Q(s,v;t)&=&\frac{q_\phi(s,v;t,1)}{s-t\,v}.
	\end{array}
	$$
Applying the results given in \cite[Theorem $2.2$]{Elk05}, we get that
\begin{equation}\label{elk}
\mbox{S}_1(p_\phi,q_\phi;s,v)=Res_{(s,v)}(P(s,v;t),Q(s,v;t))\,(s-t\,v)v^{n-3},
\end{equation}
where $\mbox{S}_1$ denotes the first subresultant polynomial operator applied to the sequence
$p_\phi,q_\phi$ with respect to the homogeneous variables $(s,v)$. By taking leading coefficients with respect to $s$ in
both sides of the latter equality, we get
$$\Delta(t,1)=\SRes(p,q)(a(t,1),b(t,1),c(t,1)) =\Res_{(s,v)}(P(s,v;t),Q(s,v;t)),
$$
the first equality is given by Theorem \ref{thm:conductor}.
\par
Now, we can decompose the map (\ref{eq:sylvestermatrix}) as the composition between:
$$\begin{array}{ccc}
	\CC[t]\otimes_\CC \CC[s,v]_{n-\mu-1} \oplus \CC[t]\otimes_\CC \CC[s,v]_{\mu-1} &\stackrel{\psi_1}{ \rightarrow} & \CC[t]\otimes_\CC \CC[s,v]_{n-2}\\ 
	 \alpha \oplus \beta  & \mapsto & \alpha \frac{p_\phi(s,v;t,1)}{s-tv}+ \beta \frac{q_\phi(s,v;t,1)}{s-tv}
\end{array}
$$
and
$$\begin{array}{ccc}
 \CC[t]\otimes_\CC \CC[s,v]_{n-2}&\stackrel{\times (s-tv)}{\rightarrow}& \CC[t]\otimes_\CC \CC[s,v]_{n-1}\\ 
\gamma&
	 \mapsto&(s-tv)\gamma.
\end{array}
$$

By setting $v=1$, it is now clear that all the minors of size $(n-1)$ in the matrix
$S(a(t,1),b(t,1),c(t,1))$ are linear combinations of maximal minors of $\psi_1$ times maximal minors of the multiplication map in the right. If we build the matrices of these morphisms using  the basis
$\{(s-tv)^js^{n-1-j}\}$ instead of  $\{s^jv^{n-1-j}\}$  and compute the matrices of these linear transformations,  then the morphism of multiplication by $s-tv$ would have an $n\times (n-1)$ matrix of
the form
$$\left(\begin{array}{cccc}
0&0&\ldots&0\\
1&0&\ldots&0\\
0&1&\ldots&0\\
\vdots&\vdots&\ddots&\vdots\\
0&0&\ldots&1\end{array}
\right).
$$
This essentially implies that in these bases there is only one nonzero minor of $\psi_1$ to be considered, and due to (\ref{elk}), it is easy to see that this minor is actually $\SRes(p,q)(a(t,1),b(t,1),c(t,1)) =\Res_{(s,v)}(P(s,v;t),Q(s,v;t))$. The proof follows straightforwardly from here.
	\end{proof}

\begin{cor}\label{cor:F1M} With the above notation, up to multiplicition by a nonzero constant in $\CC$ we have
	$d_k(t,u)=1$  for all $k>n-\mu$ and 
	$$\Delta(t,u)={d_{n-\mu}(t,u)}^{n-\mu-1}{d_{n-\mu-1}(t,u)}^{n-\mu-2}\cdots d_2(t,u)$$
\end{cor}
\begin{proof} Perform a sufficiently general linear change of coordinates in $\PP^2$ in such a way that there are no singularities of $\Cc$ at $\{x_3=0\}$, and then apply Proposition \ref{prop:F1M} and Lemma \ref{lem:chgtcoord2}.
\end{proof}

\begin{rem}
If instead of $P(s,v;t),\,Q(s,v;t)$ we used in the proof of Proposition \ref{prop:F1M}
	$\frac{F(s,v;t,1)}{s-tv},\,\frac{G(s,v;t,1)}{s-tv}$
	then we would not have
	$\Delta(t,1)$ equals to the resultant of these two polynomials anymore, as we may loose some kind of uniqueness by allowing this symmetry (see for instance the statement of Theorem
	$3.1$ in \cite{GRY02}). We will give a proper factorization of this polynomial in Section \ref{6}.	
\end{rem}

\begin{prop}\label{prop:HdivD} For any proper singularity $Q$ on $\Cc$, the polynomial $H_Q(t,u)$ defined in Proposition \ref{prop:inversionformula} divides $d_{m_Q(\Cc)}(t,u)$. Also, we have that
$H_Q(t,u)$ and $d_k(t,u)$ are coprime for all $k>m_Q(\Cc)=\deg\big(H_Q(t,u)\big)$.
\end{prop}
\begin{proof}
Let  $S_{p,q}(x_1,x_2,x_3)$ be the Sylvester matrix of the polynomials $\sum_{i=1}^3 x_ip_i(s,v)$ and $\sum_{i=1}^3 x_iq_i(s,v)$ with respect to the homogeneous variables $(s,v)$. Its entries are linear forms in $\CC[x_1,x_2,x_3]$. Therefore, its determinantal ideals, denoted $I_k(-)$, $k=1,\ldots,n$, are homogeneous ideals in $\CC[x_1,x_2,x_3]$.  
	
Then, by using Proposition \ref{prop:inversionformula} we deduce that $$V(I_k(S_{p,q}(x_1,x_2,x_3)))=\emptyset \subset \PP^2_\CC$$ for all $k=1,\ldots,\mu$, as there cannot be any common factor of degree more than $n-\mu$ of these two forms after specializing the $x_i$ (see Remark \ref{rmm}). It follows then that $$V(I_k(S_{p,q}(a(t,u),b(t,u),c(t,u))))=\emptyset \subset \PP^1_\CC$$ for all $k=1,\ldots,\mu$, and this implies $d_k(t,u)=1$ for all $k>n-\mu$.
	
Now, assume for simplicity that $Q=(0:0:1)$  as both $H_Q(t,u)$ and $d_{m_Q(\Cc)}(t,u)$ are invariant under linear changes of coordinates in $\PP^2_\CC$, and set $m:=\deg(H_Q(t,u))$. As we did above,  we have $Q\notin V(I_k(S_{p,q}(x_1,x_2,x_3)))$ for all $k=1,\ldots,n-m$ which implies that $H_Q(t,u)$ and $d_k(t,u)$ are relatively prime polynomials for all $k>m$. On the other hand, $Q\in V(I_{n-m+1}(S_{p,q}(x_1,x_2,x_3)))$, that is $I_{n-m+1}(S_{p,q}(x_1,x_2,x_3))\subset (x_1,x_2)$, and hence
	$$I_{n-m+1}(S_{p,q}(a(t,u),b(t,u),c(t,u)))\subset (a(t,u),b(t,u))\subset(H_Q(t,u)) \subset \CC[t,u]$$
	It follows that $H_Q(t,u)$ divides ${d_{n-1}(t,u)}^{n-m}\cdots {d_{m+1}(t,u)}^2d_m(t,u).$ As it is coprime with ${d_{n-1}(t,u)}^{n-m}\cdots {d_{m+1}(t,u)}^2,$ 
we conclude that  $H_Q(t,u)$ divides $d_m(t,u)$.
\end{proof}
As a corollary we recover Theorem $5$ in \cite{Chen08}.
\begin{cor}\label{cor:ordinary} The curve $\Cc$ has no infinitely near singularities if and only if for all $k=2,\ldots,n-1$
	$$d_k(t,u)=\prod_{Q \in \mathrm{Sing_p}(\Cc) \textrm{ such that } m_Q(\Cc)=k} H_Q(t,u)$$
where $\mathrm{Sing_p}(\Cc)$ denotes the subset of $\mathrm{Sing}(\Cc)$ consisting exclusively of the proper singularities of the curve $\Cc$. 
	
In particular, if the curve $\Cc$ has only ordinary singularities then Theorem \ref{thm:main} holds.
\end{cor}
\begin{proof} By Proposition \ref{prop:HdivD}, $H_Q(t,u)$ divides $d_{m_Q}(t,u)$. Note that if $Q\neq Q'$, then the polynomials $H_Q(t,u)$ and $H_{Q'}(t,u)$ are coprime
as each of them provides an inversion formula for the parameterization around two different points. We deduce from here that
$\prod_{Q\in\mathrm{Sing_p}(\Cc):\,m_Q={k}} H_Q(t,u)$ divides $d_{k}(t,u),$ and hence that 
$\prod_{Q\in\mathrm{Sing_p}(\Cc):\,m_Q={k}} H_Q(t,u)^{m_Q-1}$ divides $d_{k}(t,u)^{k-1}.$
Finally, this implies that
	$$\prod_{Q \in \mathrm{Sing_p}(\Cc)} H_Q(t,u)^{m_Q(\Cc)-1} \textrm{ divides } {d_{n-\mu}(t,u)}^{n-\mu-1}{d_{n-\mu-1}(t,u)}^{n-\mu-2}\cdots d_2(t,u)$$
	and that these two polynomials are equal (up to a nonzero multiplicative constant) if and only if for all $k=2,\ldots,n-\mu$
		$$d_k(t,u)=\prod_{Q \in \mathrm{Sing_p}(\Cc) \textrm{ such that } m_Q(\Cc)=k} H_Q(t,u).$$
		
As the polynomial $H_Q(t,u)$ has degree $m_Q(\Cc)$, we deduce that
\begin{multline}\label{eq:degHQ}
\deg\left(  \prod_{Q \in \mathrm{Sing_p}(\Cc)} H_Q(t,u)^{m_Q(\Cc)-1}  \right)= 
\sum_{Q \in \mathrm{Sing_p}(\Cc)} m_Q(\Cc)(m_Q(\Cc)-1). 
\end{multline}
But by Corollary \ref{cor:F1M} 
	$${d_{n-\mu}(t,u)}^{n-\mu-1}{d_{n-\mu-1}(t,u)}^{n-\mu-2}\cdots d_2(t,u)=\Delta(t,u),$$
and by Theorem \ref{thm:conductor}, $\Delta(t,u)$ has degree equal to \eqref{eq:degHQ} if and only if $\mathrm{Sing_p}(\Cc)=\mathrm{Sing}(\Cc),$ i.e. if and only if all 
the singularities are ordinary.
\end{proof}

\section{Proof of the main theorem}\label{4}

To prove Theorem \ref{thm:main} we will proceed by induction on the minimal length of a resolution of $\Cc$. 
Define the following property for any integer $N\geq0$:

\medskip

\noindent $(H_N)$ : {\it Theorem \ref{thm:main} holds for any rational projective plane curve $\Cc$ whose singularities can be resolved after a sequence of $N$ blow-ups, assuming that \eqref{eq:sylvestermatrix} is built from a $\mu$-basis of a proper parameterization of $\Cc$.}

\medskip

Corollary \ref{cor:ordinary} implies that the property $(H_0)$ holds. 
Now, assume that $\Cc$ can be resolved by a sequence of $N$ blow-ups and that $(H_{N-1})$ holds. By hypothesis, there exists a sequence of rational projective plane curves
$$\Cc=\Cc_0 \leftarrow \tilde{\Cc}=\Cc_1 \leftarrow \Cc_2 \leftarrow \cdots \leftarrow \Cc_{N-1}\leftarrow \Cc_N$$
such that each arrow corresponds to a blow-up (quadratic transformation) at a singular point and $\Cc_N$ has only ordinary singularities. It is clear that $\tilde{\Cc}$ can be resolved by a sequence of $N-1$ blow-ups, so Theorem \ref{thm:main} holds for $\tilde{\Cc}$ by our inductive hypothesis.

The curve $\tilde{\Cc}$ is obtained by blowing-up $\Cc$ at a singular point $P$ of $\Cc$. To simplify the notation, we will hereafter denote by $m\geq 2$ the multiplicity of $P$ on $\Cc$, by $\zf_i$ the irreducible branch-curve of $\Cc$ such that $\zf_i(t_i)=P,\,i=1,\ldots,i_P$ and by $1\leq\nu_i\leq m$ the multiplicity of $P$ on $\zf_i$. We will also denote by 
\begin{itemize}
	\item $(P=P^i_0,P^i_1,\ldots)$ the sequence of points infinitely near to $P$ in the blow-up sequence of $\zf_i$,
	\item $(m=m^i_0,m^i_1,\ldots)$ the sequence of the corresponding multiplicities of the $P^i_j$ as points on $\Cc$,
	\item and by $(\nu_i=\nu^i_0,\nu^i_1,\ldots)$ the sequence of the corresponding multiplicities as points on $\zf_i$.
\end{itemize}
Given $f(t)\in \CC[t]$ and $a\in \CC$, the notation $\val_a \big(f(t)\big)$ stands for the valuation of $f(t)$ at $a$, that is to say the largest integer $k$ such that $(t-a)^k$ divides $f(t)$.

For $i \in \{1,\ldots, i_P\}$, let $\pp_i$ be the principal ideal in $\CC[t]$ generated by $t-t_i$. Consider the $R_{\pp_i}$-module $M_{\pp_i}$. It satisfies
$\F_\nu(M)R_{\pp_i}=\F_\nu(M_{\pp_i})$ for all $\nu\in \mathbb{N}$. 
From the proof of Proposition \ref{prop:HdivD}, we already know that $\F_0(M_{\pp_i})=0$ and that $\F_j(M_{\pp_i})=R_{\pp_i}$ for all $j\geq m$, so that 
$$\val_{t_i}(d_{n-1}(t))=\val_{t_i}( d_{n-2}(t))=\cdots=\val_{t_i}( d_{m+1}(t))=0$$
Moreover, from Proposition \ref{prop:F1M} and Theorem \ref{thm:conductor} we obtain that $\F_1(M_{\pp_i})$ is generated by $d_m(t)^{m-1}d_{m-1}(t)^{m-2}\ldots d_3(t)^2d_2(t)$ and 
\begin{equation}\label{eq:F1Mlocal}
	\val_{t_i}\left( d_m(t)^{m-1}d_{m-1}(t)^{m-2}\ldots d_3(t)^2d_2(t) \right)=
\sum_{k=2}^m \left( (k-1)\sum_{m^i_j=k} \nu^i_j \right).	
\end{equation}
Recall that $\tilde{\Cc}$ is obtained after blowing up a point in $\Cc$.
 We want to assume w.l.o.g. that 
the point being blown up is $(0:0:1)$, and that  the quadratic transformation is $X=X'$ and $Y=X'Y'$. 
In order to do this correctly,  we will perform a general change of coordinates 
of $\PP^1$ and of $\PP^2$ that will simplify the blow-up computations. Recall that the $R$-module $M$
is not affected by a change of coordinates of $\PP^2$ thanks to Lemma \ref{lem:chgtcoord1}, and Lemma  \ref{lem:chgtcoord2} shows that Theorem \ref{thm:main} can be proved w.l.o.g.~in any choice of coordinates of $\PP^1$. 
\par
After then a general change of coordinates in both spaces, we can assume w.l.o.g. that

\begin{itemize}
\item[(i)]  our singular point above is $P=(0:0:1)$,
\item[(ii)] the only singularity of $\Cc$ on the line $\{x_1=0\}$ is $P$,
\item[(iii)] the line $\{x_1=0\}$ is not tangent to $\Cc$ at $P$,
\item[(iv)] $\phi(1:0)$ is not a singular point of $\Cc$,
\item[(v)] $(0:1:0)\notin \Cc$ and $(1:0:0)\notin \Cc$, i.e.~$\gcd(a,c)=\gcd(b,c)=1$,
\item[(vi)] $\phi(1:0)\in\{x_3=0\} $, which essentially means that $\deg_t(c(t,u))<n$,
\item[(vii)] there are no singularities of $\Cc$ on the line $\{x_3=0\}.$
\end{itemize}

\medskip 

Now, since we assumed (i)-(iii), we apply the quadratic transformation
$X=X'$ and $Y=X'Y'$ so that the exceptional divisor corresponds to the line $X'=0$.
The curve ${\Cc}$ is then properly parameterized on affine coordinates as follows:
$$ \AA_\CC^1 \xrightarrow{{\phi}} \AA^2_\CC :  s_0 \mapsto \left(\frac{a(s_0,1)}{c(s_0,1)},\frac{b(s_0,1)}{c(s_0,1)}\right).$$
We write 
$$\begin{array}{ccc}
a(s,1)&=&h(s)\tilde{a}(s),\\
b(s,1)&=&h(s)\tilde{b}(s),
\end{array}
$$
with $\gcd(\tilde{a},\tilde{b})=1$. Note  that with the notation of Proposition \ref{prop:inversionformula}, we have
\begin{equation}\label{hH}
h(s)=H_P(s,1)=\lambda^*\,\prod_{i=1}^{i_P}(t-t_i)^{\nu_i}
\end{equation} with $\lambda^*$ a nonzero
constant in $\CC$. By (iv), we have $m=\sum_{i=1}^{i_P}\nu_i$, and by (v), $\,\tilde{\Cc}$ has degree $2n-\nu$ and is properly
parameterized by
$$\AA_\CC^1 \xrightarrow{\tilde{\phi}} \AA^2_\CC : s_0\mapsto \left(\frac{a(s_0,1)}{c(s_0,1)},\frac{\tilde{b}(s_0)}{\tilde{a}(s_0)}\right)=
\left(\frac{a(s_0,1)\tilde{a}(s_0)}{\tilde{a}(s_0)c(s_0,1)},\frac{c(s_0,1)\tilde{b}(s_0)}{\tilde{a}(s_0)c(s_0,1)}\right).
$$
From here, it is not hard to see that a $\mu$-basis associated to $\tilde{\phi}$ is given by
$$\begin{array}{ccc}
\tilde{p}&=&\tilde{a}^h(s,v)x_2-\tilde{b}^h(s,v)x_3,\\
\tilde{q}&=&c(s,v)x_1-a(s,v)x_3,
\end{array}
$$
with $\tilde{a}^h$ and $\tilde{b}^h$ being the homogenizations of $\tilde{a}$ and $\tilde{b}$ respectively
up to degree $n-\nu$.

Now, let $\tilde{M}$ be the $R$-module built from this $\mu$-basis, i.e. 
\begin{multline*}
\tilde{M}:= 
	\coker \ \Sylv_{(s,v)}\left(\tilde{a}^h(s,v)c(t,1)\tilde{b}(t)-\tilde{b}^h(s,v)\tilde{a}(t)c(t,1),\right. \\  c(s,v)a(t,1)\tilde{a}(t) -a(s,v)\tilde{a}(t)c(t,1)\Big).
\end{multline*}
Recall from (\ref{hH}) that $H_P(s,v)$ is the homogenization of $h(s)$. As we have
\begin{multline*}
	 \tilde{a}(t)^2\left(b(s,v)c(t,1)-c(s,v)b(t,1)\right) = - \tilde{b}(t)\left(c(s,v)a(t,1)\tilde{a}(t)-a(s,v)\tilde{a}(t)c(t,1)\right)\\
	+ \tilde{a}(t)H_P(s,v)\left(\tilde{b}^h(s,v)c(t,1)\tilde{a}(t)-\tilde{a}^h(s,v)c(t,1)\tilde{b}(t)\right), 
\end{multline*}
we deduce, after setting $v=1$, that for any prime $\pp$ of $\CC[t]$ such that $\tilde{a}(t)\notin\pp$, 
the multiplication by $ b(s,1)c(t,1)-b(t,1)c(s,1)$ in the quotient ring 
$$R_\pp[s]/\big(c(s,1)a(t,1)\tilde{a}(t)-a(s,1)\tilde{a}(t)c(t,1) \big)=
R_\pp[s]/\big(c(s,1)a(t,1)-a(s,1)c(t,1)\big)$$ 
decomposes as the multiplication by $h(s)$ times the multiplication by 
$$\left(\tilde{b}(s)c(t,1)\tilde{a}(t)-\tilde{a}(s)c(t,1)\tilde{b}(t)\right)$$
(notice that since $\tilde{a}(t)\notin \pp$, one can here cancel it out without changing the valuations of the above quantities).
The leading coefficient of $a(s,1)c(t,1)-a(t,1)c(s,1)$ as a polynomial in $s$ is equal to
$a_n c(t,1)$ by (vi), and we have $c(t_i,1)\neq0$ as otherwise we will have a singularity of $\Cc$
at $\{x_3=0\}$, a contradiction with (vii). 
So, we can use the following 
\begin{lem}[{\cite[\S 3.3]{AJ06}}]\label{lem:jou} Given a commutative ring $A$ and two polynomials $f(X),g(X)$ in $A[X]$ such that the 
leading coefficient of $f$ is a unit in $A$, then the cokernel of the Sylvester matrix of $f(X)$ and $g(X)$ is isomorphic, as an $A$-module, to the cokernel of the multiplication by $g(X)$ in the quotient ring $A[X]/(f(X))$.
\end{lem} 
\noindent We deduce that for every prime $\pp_i=(t-t_i)$, the cokernel of the
localized Sylvester matrix $\tilde{M}$ is isomorphic to the cokernel of the product of the matrices associated to the  multiplication 
by $h(s)$ times the multiplication by {\small $\left(\tilde{b}(s)c(t,1)\tilde{a}(t)-\tilde{a}(s)c(t,1)\tilde{b}(t)\right)$}
in $R_{\pp_i}[s]/\big(c(s,1)a(t,1)-a(s,1)c(t,1)\big).$ 

\begin{lem} Let $\pp_i=(t-t_i)$ with $\phi(t_i:1)$ a singular point of $\Cc$. The invariant factors of the multiplication map by $h(s)$ in the quotient ring $R_{\pp_i}[s]/(a(s,1)c(t,1)-a(t,1)c(s,1))$ are
$$\alpha_1=1,\ldots,\alpha_{n-m}=1,\alpha_{n-m+1}=a(t,1),\alpha_{n-m_i+2}=a(t,1),\ldots,\alpha_{n}=a(t,1).$$
In particular, we have
$$\val_{t_i}(\alpha_k)=0,\ 1\leq k\leq n-m, \ \ \val_{t_i}(\alpha_k)=\nu_i, \ n-m+1\leq k\leq n.
$$
\end{lem}
\smallskip
\begin{proof}
As the leading coefficient of $a(s,1)c(t,1)-a(t,1)c(s,1)$ as a polynomial in $s$ is invertible $R_{\pp_i}[t]$,
then by using Lemma \ref{lem:jou}, it turns out that the cokernel of the multiplication map is actually
isomorphic to the cokernel of the following Sylvester matrix:
\begin{multline*}
	\coker\, \Sylv_s\left(a(s,1)c(t,1)-a(t,1)c(s,1),h(s)\right)\otimes R_{\pp_i}[t] \\
	\simeq\coker \Sylv_s\left(a(t,1)c(s,1),h(s)\right)\otimes R_{\pp_i}[t].	
\end{multline*}
The matrix $ \Sylv_s\left(a(t,1)c(s,1),h(s)\right)$ has $m=\deg(h)$ rows multiplied by $a(t,1)$, and
also has maximal rank as $\Res_s(c(s,1),h(s))\neq0$. This means that all the Fitting ideals of this matrix
are not zero, and moreover from $i=n+1$ to $n+h$ they are multiples of $a(t,1)^{i-1}$.
From here, the claim follows straightforwardly.
\end{proof}

Now we are ready for dealing with the inductive step and complete the proof of  the main Theorem. As we already know that the singular factors of $M$ are supported in
the singularities of $\phi$  by Proposition \ref{prop:F1M}, it is enough to prove the claim for localizations of the type $M_{\pp_0}$ with ${\pp_0}=(t-t_0),\,\phi(t_0:1)$ being a singular point of $\Cc$.
\par
Let us pick then $t_0\in\CC$ with this property. As $c(t_0,1)\neq0$ due to (vii) then, as we already used above, we have 
\begin{multline*}
	\coker\, \Sylv_s(a(s,1)c(t,1)-c(s,1)a(t,1),b(s,1)c(t,1)-c(s,1)b(t,1))\otimes
	R_{\pp_0}\\
	\simeq
	\coker\left(M_h\,M_{\tilde{b}(s)c(t,1)\tilde{a}(t)-\tilde{a}(s)c(t,1)\tilde{b}(t)}\right),
\end{multline*}
where $M_{*}$ is a matrix of the multiplication map in $R_{\pp_0}[s]/(a(s,1)c(t,1)-c(s,1)a(t,1))$.

\smallskip
If $h(t_0)\neq0$, then the character of $\phi(t_0:1)$ does not
change before and after the blow up. In addition, $M_h$ is an isomorphism and hence 
\begin{multline*}
	\coker\left(M_h\,M_{\tilde{b}(s)c(t,1)\tilde{a}(t)-\tilde{a}(s)c(t,1)\tilde{b}(t)}\right)\simeq
	\coker\left(M_{\tilde{b}(s)c(t,1)\tilde{a}(t)-\tilde{a}(s)c(t,1)\tilde{b}(t)}\right) \\
	\simeq\coker\left(\tilde{M}_{\pp_0}\right),
	\end{multline*}
$\tilde{M}$ being the matrix of the syzygies of $\tilde{\phi}$. Here, we apply the inductive
hypothesis and conclude.
\par
Suppose now that $h(t_0)=0$.  This means that $t_0\in\{t_1,\ldots, t_{i_P}\}$. Suppose w.l.o.g.
that $t_0=t_1$.

After localization, we denote 
the invariant factors of the matrix corresponding to the blow-up curve $\tilde{\Cc}$ with
\begin{multline*}
	\beta_1=1,\ldots,\beta_{d-m}=1,\beta_{d-m+1}=\tilde{d}_m,\beta_{d-m+2}=\tilde{d}_m\tilde{d}_{m-1},\ldots, \\ \beta_{d-1}=\tilde{d}_{m}\tilde{d}_{m-1}\ldots\tilde{d}_2,\beta_d=0	
\end{multline*}
and those of the matrix corresponding to $\Cc$ are set as
\begin{multline*}
	\gamma_1=1,\ldots,\gamma_{d-m}=1,\gamma_{d-m+1}=d_m,\gamma_{d-m+2}={d}_m{d}_{m-1},\ldots, \\ \gamma_{d-1}={d}_m{d}_{m-1}\ldots{d}_2,\gamma_d=0	
\end{multline*}
Applying Thompson's Theorem (Theorem \ref{thompson}), we deduce that, for all $1\leq i \leq m-1$
$$\alpha_1\ldots\alpha_{d-m+1}\beta_1\ldots\beta_{d-m}\beta_{d-m+i} \, | \, \gamma_1\ldots\gamma_{d-m}\gamma_{d-m+i},$$
that is to say
$$ (t-t_1)^{\nu_1}\tilde{d}_{m}\ldots\tilde{d}_{m-i+1}\, | \, d_m\ldots d_{m-i+1}$$
It follows that there exist non-negative integers $\epsilon_2,\ldots,\epsilon_{m-1}$ such that, for all $i=1,\ldots,m-1$, we have
$$\val_{t_1}(d_m\ldots d_{m-i+1})={\nu_1+\epsilon_{m-i+1}}+\val_{t_1}(\tilde{d}_{m}\ldots\tilde{d}_{m-i+1})$$
Therefore, we deduce that 
\begin{multline*}
	\val_{t_1}(d_m(t)^{m-1}d_{m-1}(t)^{m-2}\ldots d_3(t)^2d_2(t))= 
	\sum_{i=1}^{m-1}\val_{t_1}(d_m\ldots d_{m-i+1})= \\
	 {(m-1)\nu_1+\sum_{i=2}^m \epsilon_i}+ \val_{t_1}(\tilde{d}_m(t)^{m-1}\tilde{d}_{m-1}(t)^{m-2}\ldots\tilde{d}_3(t)^2\tilde{d}_2(t)).	
\end{multline*}
By \eqref{eq:F1Mlocal} we know that the left hand side of this equality is equal to 
$$\sum_{k=2}^m \left( (k-1)\sum_{m^1_i=k} \nu^1_i \right).$$
On the other hand, using our inductive hypothesis, the right hand side of this equality must be equal to 
$$\sum_{k=2}^m \left( (k-1)\sum_{m^1_i=k} \nu^1_i \right)+ \sum_i \epsilon_{m-i+d}.$$
Comparing the two above quantities, we deduce that 
$\sum_i \epsilon_{m-i+d}=0$
and therefore that all $\epsilon_i=0$ for all $i=2,\ldots,m$.
It follows that $d_m(t)=(t-t_1)^{\nu_1} \tilde{d_m}(t)$ and that $d_i(t)=\tilde{d_i}(t)$ for all $i=2,\ldots,m-1$. Therefore, we deduce that 
$(H_N)$ holds.



\section{Cokernels of resultant matrices}\label{5}
In this section, we show that the singular factors can be computed not only from the Sylvester matrix of the $\mu$-basis, but from a collection of matrices of smaller matrix known as the hybrid B\'ezout matrices. 

\medskip

Let $R$ be a commutative ring, $m,n\,\in\N$ with $n\geq m\geq 1,$ and 
$$\begin{array}{ccc}
f(t)&=&a_0+a_1t+\ldots+a_nt^n,\\
g(t)&=&b_0+b_1t+\ldots+b_mt^m
\end{array}
$$
polynomials in $R[t]$. 

For $k=0,\ldots, m-1$, set
$$\begin{array}{ccl}
f_k(t)&:=& a_nt^{n-m+k}+a_{n-1}t^{n-m+k-1}+\ldots+a_{m-k},\\
g_k(t)&:=&b_mt^{k}+b_{m-1}t^{k-1}+\ldots+b_{m-k},
\end{array}
$$
and define
$$
p_{k}(t):=g_k(t)f(t)-f_k(t)g(t).
$$
Note that as 
$$\begin{array}{ccl}
f(t)&=&f_k(t)t^{m-k}+a_{m-k-1}t^{m-k-1}+\ldots+a_1t+a_0\\ \\
g(t)&=&g_k(t)t^{m-k}+b_{m-k-1}t^{m-k-1}+\ldots+b_1t+b_0,
\end{array}
$$
then it turns out that $\deg(p_k(t))\leq n-1 \ \forall \, k=0,\ldots m-1$.

For  $j\in\{0,1,\ldots,m\}$, we consider the following map of $R$-modules of finite rank:
\begin{eqnarray*}
\psi_j: R^{j}\oplus R[t]_{\leq m-j-1}\oplus R[t]_{\leq n-j-1}& \to &R[t]_{\leq m+n-j-1} \\ 
\big(e_i,a(t),b(t)\big) & \mapsto & p_{m-j+i-1}(t)+a(t)f(t)+b(t)g(t)
\end{eqnarray*}
It is easy to see that $\psi_0$ is the Sylvester map of $(f,g)$. We will call $\psi_m$ the  \textit{hybrid B\'ezout} map
of these polynomials, and its matrix in the monomial bases is what we have referred to as the \textit{hybrid B\'ezout} matrix all along the text. If $n=m$, we just call them B\'ezout map and B\'ezout matrix respectively. For
$1\leq j\leq m-1, \, \psi_j$ it is also a \textit{hybrid} type map in the sense that it has a piece of Sylvester type and a piece of B\'ezout. 
\par In \cite[Proposition $18$]{AJ06} it is shown that if the leading coefficients of $f$ and $g$ generate $R$, then the cokernels of $\psi_0$ and $\psi_m$ are isomorphic. The following result is a generalization of this fact.

\smallskip
\begin{thm}\label{mt}
Suppose that there exists a nonzero divisor $d\in R,\,d\neq0$ such that $\langle a_{n},b_{m}\rangle =\langle d\rangle.$ Then the following sequence is exact
\begin{equation}\label{cond}
R/\langle d \rangle \to \mbox{coker}(\psi_{j+1})\to\mbox{coker}(\psi_j)\to R/\langle d\rangle \to0 \ \forall j=0,\ldots m-1.
\end{equation}
In particular, if $d=1$ we then have $\mbox{coker}(\psi_j)\cong\mbox{coker}(\psi_k)$ for all
$j,k$.
\end{thm}
\begin{rem}
Note that as $d\neq0$, then at least one between  $a_{n}$ and $b_{m}$ must be different from zero.
\end{rem}
\begin{proof} 
Fix $j\in\{1,\ldots m-1\}$ and consider the following commutative diagram of $R$-modules:
\begin{equation}\label{soto}
\begin{array}{ccc}
	&   &  0 \\ 
	&   &  \downarrow \\
R^{j+1}\oplus R[t]_{\leq m-j-2}\oplus R[t]_{\leq n-j-2} & \xrightarrow{\psi_{j+1}} & R[t]_{\leq m+n-j-2} \\
\\
\downarrow \alpha & & \downarrow {\bf i} \\
R^j\oplus R[t]_{\leq m-j-1}\oplus R[t]_{\leq n-j-1} & \xrightarrow{\psi_{j}} & R[t]_{\leq m+n-j-1} \\
\downarrow & & \downarrow \\
\mbox{coker}(\alpha) & \xrightarrow{\beta} &  \mbox{coker}({\bf i})\\
\downarrow & & \downarrow \\
0 & & 0
\end{array}
\end{equation}
where $\beta$ is the induced morphism of cokernels, ${\bf i}$ the canonical injection, and  $\alpha$ is defined as
\begin{equation}\label{alpha}
\begin{array}{lclc}
\alpha(0,a(t),b(t))&=&(0,a(t),b(t))& \\
\alpha(e_1,0,0)&=&(0,g_{m-j-1}(t),-f_{m-j-1}(t))& \\
\alpha(e_{i+1},0,0)&=&(e_i,0,0)& \ \mbox{for }\,i>1.
\end{array}
\end{equation}
Clearly, we have $\mbox{coker}({\bf i})\cong t^{m+n-j-1}\,R$. On the other hand, it is easy to see that
\begin{equation}\label{1:1}
\mbox{im}(\alpha)\cong R^j\oplus R[t]_{\leq m-j-2}\oplus R[t]_{\leq n-j-2}\oplus
(b_mt^{m-j-1},-a_nt^{n-j-1})R.
\end{equation}
Let $u,v,w,z\in R$ such that
$ua_n+vb_m=d,\,wd=a_n,\, zd=b_m$. As $d$ is not a zero divisor in $R$, then we have
$uw+vz=1$, and hence we have an isomorphism
$$(t^{m-j-1},0)R\oplus(0,t^{n-j-1})R\cong(ut^{m-j-1},vt^{n-j-1})R\oplus(zt^{m-j-1},-wt^{n-j-1})R,
$$
and from here, using (\ref{1:1}),
we then get
$$\mbox{coker}(\alpha)\cong (ut^{m-j-1},vt^{n-j-1})R\oplus \big(zt^{m-j-1},-wt^{n-j-1}\big)R/\langle d\rangle.
$$
With these identifications, it is straightforward to compute $\beta$ explicitly:
$$\beta\left(r_1(ut^{m-j-1},vt^{n-j-1})+[r_2]\big(zt^{m-j-1},-wt^{n-j-1}\big)
\right)=dr_1t^{m+n-j-1}
$$
for $r_1\in R,\, [r_2]\in R/\langle d \rangle.$ We deduce, then
\begin{equation}\label{beta}
\begin{array}{ccc}
\mbox{ker}(\beta)&\cong & R/\langle d\rangle, \\
\mbox{coker}(\beta)&\cong & R/\langle d\rangle.
\end{array}
\end{equation}
The claim now follows straightforwardly by applying the Snake Lemma to (\ref{soto}).
\end{proof}

The following lemma will imply that the cokernels of all the hybrid matrices of $\mu$-bases are
isomorphic and hence that the Fitting invariants of all type of resultant matrices are the same.
\begin{lem}\label{lemutil}
Let $p_\phi(s,1;t),\,q_\phi(s,1;t)\in\CC[s,t]$ be the specialization in $v=1$ of the polynomials defined in
(\ref{pq}). If $\phi(1:0)$ is not a singular point on $\Cc$, then the leading coefficients of these two polynomials with respect to $s$ are coprime in
$\CC[t]$.
\end{lem}

\begin{proof}
Note that the leading coefficients of $\sum_{i=1}^3 p_i(s,1)x_i$ and $\sum_{i=1}^3 q_i(s,1)x_i$ are two linear forms, say $L_p(x_1,x_2,x_3)$ and $L_q(x_1,x_2,x_3)$ in $\CC[x_1,x_2,x_3]$ that are $\CC$-linearly independent, otherwise  by making a reversible linear combination of these elements, one could replace 
the $\mu$-basis $(p_\phi,q_\phi)$ with $(p_\phi,q'_\phi)$ with $\deg(q'_\phi)< n-\mu$, a
contradiction. 
\par Now, it is easy to see that these two linear forms intersect at $\phi(1:0)$, and as they are linearly independent, this is their only
point of intersection in $\PP^2$. From here
it follows that $L_p(a(t,1),b(t,1),c(t,1))$ and $L_q(a(t,1),b(t,1),c(t,1))$ are necessarily coprime in $\CC[t]$, otherwise
they will have a common factor $h(t)$ of positive degree, and we would have that
$$\phi(t_0:1)=\big(a(t_0,1):b(t_0,1):c(t_0:1)\big)=\phi(1:0)$$
for every $t_0$ such that $h(t_0)=0$, contradicting the fact that $\phi(1:0)$ is not a singularity of $\Cc$.
\end{proof}

\begin{cor} The singular factors of all the matrices $\psi_j\big(p_\phi(s,v;t,u),q_\phi(s,v;t,u)\big)$ in $\CC[t,u]$ are the same for any $j=0,\ldots, m,$ and  for any $\{p,q\}$ $\mu$-basis of $\phi$.
\end{cor}
\begin{proof}  If $\phi(1:0)$ is not a singular point on $\Cc$, then the result follows from Theorem \ref{mt} and Lemma \ref{lemutil}. If this is not the case, by applying a linear change of coordinates in $\PP^1$, we may assume that $\phi(1:0)$ is not a singular point on $\Cc$ and the result then follows from Lemma \ref{lem:chgtcoord2} which holds not only for the Sylvester matrix $\psi_0$, but also for all the matrices $\psi_j$, $j=0,\ldots,m$ (the same proof works verbatim).
\end{proof}

\section{On the invariant factors of the $D$-resultant matrix}\label{6}
In this section, we will describe the invariant factors of a matrix closely related to $S_{p_\phi,q_\phi}(t,u)$ that was originally studied in \cite{CS01} in order to compute the singularities of $\Cc$. As a consequence we obtain a complete factorization of the $D$-resultant for rational polynomials, introduced in \cite{GRY02}.

\smallskip
Let $B(x_1,x_2,x_3)\in\CC[x_1,x_2,x_3]^{n\times n}$ be the B\'ezout matrix associated to the polynomials
$a(s,v)x_3-c(s,v)x_1$ and $b(s,v)x_2-c(s,v)x_3$ with respect to the homogeneous variables $(s,v)$, and
$S(x_1,x_2,x_3)\in\CC[x_1,x_2,x_3]^{n\times n}$ be the Sylvester matrix associated to  $p(s,v)=\sum_{i=1}^3 x_ip_i(s,v)$ and $q(s,v)=\sum_{i=1}^3 x_iq_i(s,1)$ with respect to $(s,v)$. Here, as usual, $\{p,q\}$ is a
$\mu$-basis of $(a,b,c)$.

\begin{prop}\label{kk}
There exists an invertible $N\in\CC^{n\times n}$  such that 
\begin{equation}\label{level1}
B(x_1,x_2,x_3)= x_3\,N\,S(x_1,x_2,x_3).
\end{equation}
\end{prop}

\begin{proof}
Set $B=\big(B_{i,j}(x_1,x_2,x_3)\big)_{0\leq i,j\leq n-1}$. We then have
$$\begin{array}{ccl}
	\sum_{i=0}^{n-1}\sum_{j=0}^{n-1}B_{i,j}(x_1,x_2,x_3)s^it^j &=& 
\frac{1}{s-t}\left(\big(a(s,1)x_3-c(s,1)x_1\big)\big(b(t,1)x_3-c(t,1)x_2\big)\right. \\ 
&&\left.-\big(a(t,1)x_3-c(t,1)x_1\big)
\big(b(s,1)x_3-c(s,1)x_2\big)\right).
\end{array}
$$
By setting $x_3=0$ above, it is easy to see that the right hand side vanishes, and hence we have $B_{i,j}(x_1,x_2,0)=0$ for all $i,j$. This shows that 
$$B_{i,j}(x_1,x_2,x_3)=x_3A_{i,j}(x_1,x_2,x_3) \ i,j=0,\ldots, n-1
$$ 
with $A_{i,j}(x_1,x_2,x_3)$ a homogeneous linear form.
If now we substitute $$x_1\mapsto a(s,1),\,x_2\mapsto b(s,1),\,x_3\mapsto c(s,1)$$ we again have that the whole
Bezoutian polynomial vanishes. So we conclude that
$$c(s,1)\sum_{i=0}^{n-1}\left(\sum_{j=0}^{n-1}A_{i,j}(a(s,1),b(s,1),c(s,1))s^i\right)t^j=0.
$$
As $c(s,1)\neq0$, this shows that for all $i=0,\ldots, n-1$,
$$L_i(s,v;x_1,x_2,x_3):=\sum_{j=0}^{n-1}A_{i,j}(x_1,x_2,x_3)s^iv^{n-1-i}$$ is a syzygy of $(a,b,c)$ of degree $n-1$. Moreover, the fact that $\det(B)\neq0$ (as we have assumed $\gcd(a,b,c)=1$)
shows then that $\det(A_{ij}(x_1,x_2,x_3))\neq0$ and this implies that the family
$\{L_0,\ldots, L_{n-1}\}$ is a basis of the $\CC$-vector space of syzygies of $(a,b,c)$ of degree $n-1$.
\par
On the other hand, it is easy to check that the family $$\{v^{n-\mu-1}p,v^{n-\mu-2}sp,\ldots, s^{n-\mu-1}p,v^{\mu-1}q,v^{\mu-2}sq,\ldots,s^{\mu-1}q\}$$ is another basis of the same $\CC$-vector space.  This is due to the fact that the matrix of coefficients of
this family with respect to the monomial basis is actually $S(x_1,x_2,x_3)$, whose determinant gives 
the implicit equation.
\par
So, as both sets are bases of the same space, we then get that there exists an invertible $N\in\CC^{n\times n}$ such that
$$\big(A_{i,j}(x,y,z)\big)=N\,S.
$$
From here, the proof follows straightforwardly.
\end{proof}
As a direct application of Proposition \ref{kk} we get the explicit description of the invariant
factors of the matrix $B_{F,G}(t,u)$ stated in the introduction.
Denote with $D_i(B_{F,G})$ the $\gcd$ of the $(n-i)$-minors of $B_{F,G}(t,u)$. 
\begin{thm}\label{cor:tt}
$D_0(B_{F,G})=0$ and for $i=1,\ldots,n-1$,
$$D_i(B_{F,G})=c(t,u)^{n-i}\,d_n(t,u)^{n-i} d_{n-1}(t,u)^{n-i-1}\ldots d_{i+1}(t,u).
$$
\end{thm}
\begin{proof}
Set $x_1\mapsto a(t,u),\,x_2\mapsto b(t,u),\, x_3\mapsto c(t,u)$ in (\ref{level1}) and compute the invariant factors on both sides.
\end{proof}

\medskip

As an immediate consequence, we also get the following
\begin{thm}[Factorization of the $D$-resultant, case of same denominator]\label{thm:dres} 
\begin{equation}\label{factoriz}
\tilde{\Delta}(t,u)=c(t,u)^{n-1}\,d_n(t,u)^{n-1} d_{n-1}(t,u)^{n-2}\ldots d_{2}(t,u).
\end{equation}
\end{thm}
\begin{proof}
Recall that $\Res_{(s,v)}\left(\frac{F(s,v;t,u)}{su-tv},\frac{G(s,v;t,u)}{su-tv}\right)$ is equal to the first subresultant of the pair $F(s,v;t,u),\,G(s,v;t,u)$. Set $i=1$ in Theorem \ref{cor:tt} and use Proposition \ref{prop:F1M}.
\end{proof}

Actually, the $D$-resultant  in \cite{GRY02} was defined for an affine parameterization of the
form $\big(\frac{A(t)}{C(t)},\frac{B(t)}{D(t)}\big)$ with $\gcd(A,C)=\gcd(B,D)=1.$ In order to tackle this situation, set 
$$ n_1:=\max\{\deg(A),\,\deg(C)\},\ n_2:=\max\{\deg(B),\,\deg(D)\}.
$$
Let $\tilde{A}(s,v),\,\tilde{C}(s,v)$ (resp. $\tilde{B}(s,v),\,\tilde{D}(s,v)$) be the homogenizations of
$A$ and $C$ (resp. $B$ and $D$) to degree $n_1$ (resp. $n_2$). The $D$-resultant of the curve given by this parameterization is defined in \cite{GRY02} as
\begin{multline}\label{tilded}
\tilde{\Delta}_{\tilde{A},\tilde{C},\tilde{B},\tilde{D}}(t,u):= \\
\Res_{(s,v)}\left(
\frac{\tilde{A}(s,v)\tilde{C}(t,u)-\tilde{A}(t,u)\tilde{C}(s,v)}{su-tv},
\frac{\tilde{B}(s,v)\tilde{D}(t,u)-\tilde{B}(t,u)\tilde{D}(s,v)}{su-tv}
\right).
\end{multline}
Denote with $\tilde{\Cc}\subset\PP^2$ the curve defined by the closure of the image of the
parameterization  given by $\big(\frac{A(t_0)}{C(t_0)},\frac{B(t_0)}{D(t_0)}\big)$, with
$t_0\in\CC$. We assume that this parameterization is proper, and hence $\tilde{\Cc}$ is birationally parametrized by
$$
\begin{array}{cccc}
\nu:&\PP^1&\to&\PP^2\\
&(s_0:v_0)&\mapsto&\big(\tilde{a}(s_0:v_0):\tilde{b}(s_0:v_0):\tilde{c}(s_0:v_0)\big),
\end{array}
$$
with $\tilde{c}(s,v)$ being  the least common multiple
of $\tilde{C}(s,v)$ and $\tilde{D}(s,v)$; $\tilde{a}(s,v):=\frac{\tilde{A}(s,v)\tilde{c}(s,v)}{\tilde{C}(s,v)}$ and  $\tilde{b}(s,v):=\frac{\tilde{B}(s,v)\tilde{c}(s,v)}{\tilde{D}(s,v)}$. The polynomials $\tilde{a},\,\tilde{b},\,\tilde{c}$ have then the same degree $n\geq\max\{n_1,n_2\}$, and no common factors. Hence, the degree of $\tilde{\Cc}$ is then $n$ and we have
$$\tilde{c}(t,u)=h(t,u)\tilde{C}(t,u)=q(t,u)\tilde{D}(t,u),$$ with $h(t,u)$ and $q(t,u)$ coprimes. We also get $\tilde{a}(t,u)=h(t,u)\tilde{A}(t,u)$ and $\tilde{b}(t,u)=q(t,u)\tilde{B}(t,u)$, and $\gcd(\tilde{a}(t,u),\tilde{b}(t,u),\tilde{c}(t,u))=1.$ 

We will denote with  $\Delta_\nu(t,v)$ the polynomial defined in (\ref{eq:delta}) associated with the
parameterization $\nu$. A complete factorization of this polynomial in terms of the singularities of $\tilde{\Cc}$ and its multiplicity graph is given in Corollary \ref{cor:F1M}.

\par
Finally, let $\delta(t,u):=\gcd(B(t,u),D(t,u))$. Note that we have
$$\tilde{C}(t,u)=q(t,u)\delta(t,u),\ \ \tilde{D}(t,u)=h(t,u)\delta(t,u).
$$

\begin{thm}[Factorization of the $D$-resultant, case of different denominators]\label{thm:diffd}
If the parameterization defined by
$\big(\frac{A(t)}{C(t)},\frac{B(t)}{D(t)}\big)$ is proper, then with the notation established above, we have
\begin{equation}\label{factorizz}
{h(t,u)}^{\deg(h)-1}{q(t,u)}^{\deg(q)-1}\tilde{\Delta}_{\tilde{A},\tilde{C},\tilde{B},\tilde{D}}(t,u)={\delta(t,u)}^{\deg(\delta)-1}\Delta_{\nu}(t,u).
\end{equation}
\end{thm}

\begin{rem}	Note that (\ref{factorizz}) generalizes (\ref{factoriz}), as in the case of common denominators we have $h(t,u)=q(t,u)=1$ and $\delta(t,u)=\tilde{c}(t,u)$.
\end{rem}

\begin{proof}
We apply Theorem \ref{thm:dres}
to the parameterization given by  $(\tilde{a}:\tilde{b}:\tilde{c})$ and have
\begin{equation}\label{eq:ddd}
\tilde{\Delta}_{\tilde{a},\tilde{b},\tilde{c}}(t,u)=\tilde{c}(t,u)^{n-1}\Delta_{\nu}(t,u),
\end{equation}
where $\tilde{\Delta}_{\tilde{a},\tilde{b},\tilde{c}}(t,u):=\Res_{(s,v)}\left(
\frac{\tilde{a}(s,v)\tilde{c}(t,u)-\tilde{a}(t,u)\tilde{c}(s,v)}{su-tv},
\frac{\tilde{b}(s,v)\tilde{c}(t,u)-\tilde{b}(t,u)\tilde{c}(s,v)}{su-tv}
\right)$, which actually factorizes as

$$
\begin{array}{cl}
&\Res_{(s,v)}\left(
h(s,v)h(t,u)\frac{\tilde{A}(s,v)\tilde{C}(t,u)-\tilde{A}(t,u)\tilde{C}(s,v)}{su-tv},
q(s,v)q(t,u)\frac{\tilde{B}(s,v)\tilde{D}(t,u)-\tilde{B}(t,u)\tilde{D}(s,v)}{su-tv}
\right)\\ \\
=&\lambda_0\,h(t,u)^{n-1}q(t,u)^{n-1}\Res_{(s,v)}\left(h(s,v),\frac{\tilde{B}(s,v)\tilde{D}(t,u)-\tilde{B}(t,u)\tilde{D}(s,v)}{su-tv}\right)\times \\
 &  \hspace{3cm} \Res_{(s,v)}\left(\frac{\tilde{A}(s,v)\tilde{C}(t,u)-\tilde{A}(t,u)\tilde{C}(s,v)}{su-tv},q(s,v)\right)  \tilde{\Delta}_{\tilde{A},\tilde{C},\tilde{B},\tilde{D}}(t,u),
\end{array}
$$
with $\lambda_0:=\Res_{(s,v)}\big(h(s,v),q(s,v)\big)\neq0.$
As
$\frac{\tilde{B}(s,v)h(t,u)\delta(t,u)-\tilde{B}(t,u)h(s,v)\delta(s,v)}{su-tv}$ can be written as $$\delta(t,u)\tilde{B}(s,v)\frac{h(t,u)-h(s,v)}{su-tv}+h(s,v)\frac{\tilde{B}(s,v)\delta(t,u)-\tilde{B}(t,u)\delta(s,v)}{su-tv},
$$
and using the fact that  
$D(t,u)=h(t,u)\delta(t,v)
$ we get that 
{\scriptsize $$\Res_{(s,v)}\left(h(s,v),\frac{\tilde{B}(s,v)\tilde{D}(t,u)-\tilde{B}(t,u)\tilde{D}(s,v)}{su-tv}\right)=\Res_{(s,v)}\left(h(s,v),\delta(t,u)\tilde{B}(s,v)\frac{h(t,u)-h(s,v)}{su-tv}\right).$$}
Note that $(\frac{\tilde{A}h}{\tilde{C}h},\frac{\tilde{B}q}{\tilde{D}q})$ is the minimal expression that makes the denominators $\tilde{C}h=\tilde{D}q$,
hence $h$ must be coprime with $\tilde{B}$ otherwise the second fraction would simplify. So, we have
$\Res_{(s,v)}(h(s,v),\tilde{B}(s,v))=\lambda^*\neq0$, and then 
{\scriptsize $$\Res_{(s,v)}\left(h(s,v),\frac{\tilde{B}(s,v)\tilde{D}(t,u)-\tilde{B}(t,u)\tilde{D}(s,v)}{su-tv}\right)=\lambda^*\delta(t,u)^{\deg(h)}\Res_{(s,v)}\big(h(s,v),\frac{h(t,u)-h(s,v)}{su-tv}\big).
$$}
By using the Poisson formula for the resultant, we have that -up to a nonzero constant-
$$\Res_{(s,v)}\left(h(s,v),\frac{h(t,u)-h(s,v)}{su-tv}\right)=\prod_{h(\xi_0:\xi_1)=0}\frac{h(t,u)}{\xi_0u-\xi_1t}={h(t,u)}^{\deg(h)-1}.
$$
So,  we get
$$\Res_{(s,v)}\left(h(s,v),\frac{\tilde{B}(s,v)\tilde{D}(t,u)-\tilde{B}(t,u)\tilde{D}(s,v)}{su-tv}\right)=\lambda_1\,\delta(t,u)^{\deg(h)}h(t,u)^{\deg(h)-1},
$$
with $\lambda_1\in\CC_{\neq0}$.
\par\smallskip
The computation of $\Res_{(s,v)}\left(\frac{\tilde{A}(s,v)\tilde{C}(t,u)-\tilde{A}(t,u)\tilde{C}(s,v)}{su-tv},q(s,v)\right)$ follows the same line: one has that -up to a nonzero constant- $\tilde{C}(t,u)=q(t,u)\delta(t,u)$ and then
$$\Res_{(s,v)}\left(\frac{\tilde{A}(s,v)\tilde{C}(t,u)-\tilde{A}(t,u)\tilde{C}(s,v)}{su-tv},q(s,v)\right)=\lambda_2\,{\delta(t,u)}^{\deg(q)}{q(t,u)^{\deg(q)-1}},
$$
for $\lambda_2\neq0$. Collecting all this information, we get
$$
\tilde{\Delta}_{\tilde{a},\tilde{b},\tilde{c}}(t,u)=
\lambda\ {\tilde{c}(t,u)}^{n-\deg(\delta)}{h(t,u)}^{n-\deg(q)-2}{q(t,u)}^{n-\deg(h)-2} \tilde{\Delta}_{\tilde{A},\tilde{C},\tilde{B},\tilde{D}}(t,u)
$$
with $\lambda\neq0$.
And now we use (\ref{eq:ddd}) to get that -up to a constant-
{\small $$\begin{array}{ccl}
{h(t,u)}^{n-\deg(q)-2}{q(t,u)}^{n-\deg(h)-2} \tilde{\Delta}_{\tilde{A},\tilde{C},\tilde{B},\tilde{D}}(t,u)&=&\tilde{c}(t,u)^{\deg(\delta)-1}\Delta_{\nu}(t,u)\\
&=&\big(h(t,u)q(t,u)\delta(t,u)\big)^{\deg(\delta)-1}\Delta_{\nu}(t,u).
\end{array}$$}
From here, we deduce
$$ {h(t,u)}^{\deg(h)-1}{q(t,u)}^{\deg(q)-1}\tilde{\Delta}_{\tilde{A},\tilde{C},\tilde{B},\tilde{D}}(t,u)={\delta(t,u)}^{\deg(\delta)-1}\Delta_{\nu}(t,u),
$$
which is the claim we wanted to prove.
\end{proof}


\end{document}